\documentclass[11pt]{article}

\usepackage[english]{babel}
\usepackage{amssymb}
\usepackage{amsmath}
\usepackage{enumerate}
\usepackage{amsthm}
\usepackage{amsfonts}
\usepackage{listings}
\usepackage{xypic}
\usepackage{hyperref}

\usepackage{MnSymbol}

\title{Exactness of Bousfield localizations of simplicial presheaves and local lifting}
\date{April 2021}
\author{Fritz H\"ormann\\ Mathematisches Institut, Albert-Ludwigs-Universit\"at Freiburg}

\usepackage[numbers,sort&compress]{natbib}

\usepackage{color}

\usepackage{tikz}

\newcommand{\comment}[1]{}

\newtheorem{SATZ}{Theorem}[section]

\newtheorem{LEMMA}[SATZ]{Lemma}
\newtheorem{DEF}[SATZ]{Definition}
\newtheorem{PROP}[SATZ]{Proposition}

\newtheorem{BEISPIEL}[SATZ]{Example}

\newtheorem{KOR}[SATZ]{Corollary}

\newtheorem{BEM}[SATZ]{Remark}

\newtheoremstyle{bare}        % name
  {}            % Space above, empty = `usual value'
  {}            % Space below
  {\normalfont}                 % Body font (\normalfont)
  {}                            % Indent amount (empty = no indent, \parindent = para indent)
  {\bfseries}                   % Thm head font
  {}                            % Punctuation after thm head
  {.0em}                           % Space after thm head: " " = normal interword space;
  {\thmnumber{#2}#1. \thmnote{\normalfont\textsc{(#3)}} } % Thm head spec

\theoremstyle{bare}
\newtheorem{PAR}[SATZ]{}

\DeclareMathOperator{\Fib}{Fib}
\DeclareMathOperator{\Cof}{Cof}

\DeclareMathOperator{\colim}{colim}

\DeclareMathOperator{\Hom}{Hom}

\DeclareMathOperator{\op}{op}

\begin{document}

\maketitle



\section*{Abstract}
We show that weak equivalences in a (cofibrantly generated) left Bousfield localization of the projective model category of simplicial presheaves can be characterized by a local lifting property
if and only if the localization is exact. 

\tableofcontents

\section{Introduction}

Let $\mathcal{S}$ be a small category with Grothendieck topology. 
Dugger, Hollander and Isaksen \cite{DHI04} showed as a byproduct of their proofs that weak equivalences in the Bousfield localization of simplicial presheaves on $\mathcal{S}$ at all hypercovers
can be characterized by a local lifting property which itself involves {\em hypercovers} of objects in $\mathcal{S}$ and their refinements. 
This article grew out of an attempt to generalize a similar statement to \v{C}ech weak equivalences (i.e.\@ weak equivalences in the localization at the \v{C}ech covers) with
the hope to get a better understanding of these. It is shown that such a characterization by local lifting is possible more generally --- by purely formal reasons --- whenever the
localization is exact (i.e.\@ as functor of derivators, or $(\infty, 1)$-categories, commutes with homotopically finite homotopy limits).% as are the localizations at all \v{C}ech covers and at all hypercovers. 
% Unfortunately the result does not shed much light on the nature of \v{C}ech weak equivalences because the characterization is quite self. 

The main result is the following: 

\vspace{0.3cm}

{\bf Theorem~\ref{MAINTHEOREMEXACT}.} 
Let $\mathcal{S}$ be a small category. 
Choose the projective model category structure on $\mathcal{SET}^{\mathcal{S}^{\op} \times \Delta^{\op}}$. Consider a (cofibrantly generated) left Bousfield localization with class $\mathcal{W}_{loc}$ of weak equivalences. 
Let $\mathcal{COV}$ be a subcategory of coverings satisfying (C1)--(C4) below.
The following are equivalent:
\begin{enumerate}
\item $\mathcal{W}_{loc}$ is stable under pull-back along fibrations;
\item $S$, a generating set of cofibrations, goes to $\mathcal{W}_{loc}$ under pull-back along fibrations with cofibrant source;
\item The left Bousfield localization is exact, i.e.\@ the localization functor (left adjoint) commutes with homotopically finite homotopy limits;
\item $\mathcal{W}_{loc} \subset \mathcal{W}_{\mathcal{COV}}$;
\item $\mathcal{W}_{loc} = \mathcal{W}_{\mathcal{COV}}$.
\end{enumerate}

\vspace{0.3cm}

In the Theorem $\mathcal{COV} \subset \mathcal{SET}^{\mathcal{S}^{\op} \times \Delta^{\op}}$ is a subcategory of ``coverings'' which satisfies the following axioms:
\begin{enumerate}
\item[(C1)] Each object of $\mathcal{COV}$ is cofibrant;
\item[(C2)] Every representable presheaf (considered as constant simplicial presheaf) is in $\mathcal{COV}$;
%\item[(C3)] Trivial fibrations are in $\mathcal{COV}$.

\item[(C3)] Each morphism of $\mathcal{COV}$ is in $\mathcal{W}_{loc}$;
\item[(C4)] If $Y \in \mathcal{COV}$ and $Y' \rightarrow Y$ is in $\Fib \cap \mathcal{W}_{loc}$ with $Y'$ cofibrant then $Y' \rightarrow Y$ is in $\mathcal{COV}$.
\end{enumerate}

The class $\mathcal{W}_{\mathcal{COV}}$ consists by definition of those morphisms $f$ for which there is a diagram
\[ \xymatrix{
X \ar[r]^{\mathcal{W}}  \ar[d]_f & X' \ar[d]^{f'} \\
Y \ar[r]_{\mathcal{W}} & Y'
}\]
where $f'$ has the local homotopy lifting property w.r.t.\@ $\mathcal{COV}$ (cf.\@ Definition~\ref{DEFHLP}), $X'$ and $Y'$ are 
%locally fibrant (w.r.t.\@ $\mathcal{COV}$, cf.\@ Definition~\ref{DEFLOCFIB}) 
fibrant and the horizontal morphisms are in $\mathcal{W}$. 

The $(\infty, 1)$-category defined by a cofibrantly generated exact localization of $\mathcal{SET}^{\mathcal{S}^{\op} \times \Delta^{\op}}$ as in Theorem~\ref{MAINTHEOREMEXACT} is
by definition an {\bf $(\infty,1)$-topos}. If the localization is furthermore {\bf topological} in the sense of Lurie \cite[Definition~6.2.1.4]{Lur09}, it can be shown that
it corresponds to a Grothendieck topology on $\mathcal{S}$ and is given by the \v{C}ech localization considered in Example~\ref{EX1}. However, also the localization at all hypercovers, which presents
the hypercompletion of this $(\infty,1)$-topos is still exact.

\section*{Notation}

Let $\mathcal{S}$ be a small category.
Recall that the category $\mathcal{SET}^{\mathcal{S}^{\op} \times \Delta^{\op}}$ of
 simplicial presheaves on $\mathcal{S}$ is a simplicial category which is tensored and cotensored. We denote the corresponding 
 functors by 
 \[ \otimes:  \mathcal{SET}^{\Delta^{\op}} \times \mathcal{SET}^{\mathcal{S}^{\op} \times \Delta^{\op}} \rightarrow \mathcal{SET}^{\mathcal{S}^{\op} \times \Delta^{\op}}, \]
 and
 \[ \Hom:  (\mathcal{SET}^{\Delta^{\op}})^{\op} \times \mathcal{SET}^{\mathcal{S}^{\op} \times \Delta^{\op}} \rightarrow \mathcal{SET}^{\mathcal{S}^{\op} \times \Delta^{\op}}. \]
 For a morphism $f: X\rightarrow Y$ of simplicial sets  and a morphism $g: A \rightarrow B$ of simplicial presheaves we denote by
 \[ f \boxplus g: (X \otimes B) \oplus_{(X \otimes A)} (Y \otimes A) \rightarrow (Y \otimes B)  \]
 the induced morphism and likewise
 \[ \boxdot \Hom(f, g): \Hom(Y, A) \rightarrow \Hom(Y, B) \times_{\Hom(X, B)} \Hom(X, A).  \]

%The aim of this article is to understand un

%A higher topos can be represented as localization of ...

\section{Exactness of left Bousfield localizations}

In this section let $(\mathcal{M}, \Cof, \Fib, \mathcal{W})$ be a right proper model category and
let $(\mathcal{M}, \Cof, \Fib_{loc}, \mathcal{W}_{loc})$ be a left Bousfield localization thereof.

\begin{DEF}\label{DEFP}
A morphism $f$ has {\bf property $\mathbf{P}$} if any pull-back of $f$ is in $\mathcal{W}_{loc}$. A morphism $f$ has {\bf property $\mathbf{P}_{fib}$} if any pull-back of $f$ along a fibration is in $\mathcal{W}_{loc}$.
\end{DEF}
Obviously we have 
\[ f \text{ has } \mathbf{P} \Rightarrow f \text{ has } \mathbf{P}_{fib} \Rightarrow  f \in \mathcal{W}_{loc}  \]
and also (right properness and $\mathcal{W} \subset \mathcal{W}_{loc}$):
\[ f \in \mathcal{W} \Rightarrow  f \text{ has } \mathbf{P}_{fib}.  \]

\begin{LEMMA}\label{LEMMAP1}
\begin{enumerate}
\item If a fibration has property $\mathbf{P}_{fib}$ then it has property $\mathbf{P}$.
\item If $f = hg$ and $g$ has property $\mathbf{P}_{fib}$ then $f$ has property $\mathbf{P}_{fib}$ if and only if $h$ has property $\mathbf{P}_{fib}$.
\item If $f$ is a morphism with property $\mathbf{P}_{fib}$ and $f = h g$ with $g$ trivial cofibration and
$h$ fibration, then $h$ has property $\mathbf{P}$.
\end{enumerate}
\end{LEMMA}
\begin{proof}
1. Let $w: X \rightarrow Y$ be a fibration with property $\mathbf{P}_{fib}$. 
Let $f$ be an arbitrary morphism. Factor $f = pc$ with $c \in \Cof \cap \mathcal{W}$ and $p \in \Fib$. 
Then in the pull-back
\[ \xymatrix{
\Box  \ar[r]^{c'} \ar[d]_{w''} & \Box \ar[r] \ar[d]^{w'}  & X \ar[d]^{w} \\
Z \ar[r]_c  & \ar[r]_p Z' & Y
}\]
the morphism $w'$ is in $\mathcal{W}_{loc}$ by assumption. It is also in $\Fib$, hence by right properness we get $c' \in \mathcal{W} \subset \mathcal{W}_{loc}$. 
Therefore $w'' \in \mathcal{W}_{loc}$ by 2-out-of-3. 

2. Consider a diagram in which the squares are Cartesian
\[ \xymatrix{
X' \ar[r]^{\Fib} \ar[d]_{\mathcal{W}_{loc}} & X \ar[d]^g \\
Y' \ar[r]_{\Fib} \ar[d] & Y \ar[d]^h \\
Z' \ar[r]_{\Fib} & Z 
}\]
Since $g$ has property $\mathbf{P}_{fib}$ the upper left morphism is in $\mathcal{W}_{loc}$. Hence the statement follows from 2-out-of-3. 

3. $g$ has property $\mathbf{P}_{fib}$ and thus by 2.\@ the same holds for $h$. Therefore, we deduce from 1.\@ that $h$ has property $\mathbf{P}$. 
\end{proof}

\begin{SATZ}\label{SATZPROPEREXACTLOC}
Let $(\mathcal{M}, \Cof, \Fib, \mathcal{W})$ be a right proper model category and
let  $(\mathcal{M}, \Cof, \Fib_{loc}, \mathcal{W}_{loc})$ be a left Bousfield localization. Then the following are equivalent:
\begin{enumerate}
\item $\mathcal{W}_{loc}$ is stable under pull-back along morphisms in $\Fib$;
\item The localization functor (left adjoint) commutes with homotopy
pull-backs (and hence with homotopically finite homotopy limits). We also say that the localization is {\bf  exact}.
\end{enumerate}
\end{SATZ}
\begin{proof}[Proof (compare also {\cite[6.2.1.1]{Lur09}})]
$1. \Rightarrow 2.:$ Assume that $\mathcal{W}_{loc}$ is stable under pull-back along morphisms in $\Fib$, i.e.\@ all morphisms in $\mathcal{W}_{loc}$ have property $\mathbf{P}_{fib}$. Then by Lemma~\ref{LEMMAP1}, 1.\@ the class $\Fib \cap \mathcal{W}_{loc}$ is stable under arbitrary pull-back.
Consider a morphism of diagrams
\[ \mu : \vcenter{\xymatrix{ 
& X \ar[d] \\
Z \ar[r] & Y } } \rightarrow  \vcenter{\xymatrix{ 
& X' \ar[d] \\
Z' \ar[r] & Y' } }
 \]
in which in both diagrams both morphisms are in $\Fib$, all objects are $\Fib$-fibrant, and $\mu$  is point-wise in $\mathcal{W}_{loc}$. 
We claim that then the induced morphism between pull-backs is in $\mathcal{W}_{loc}$. 
We may factor $\mu = \mu_2 \mu_1$ where $\mu_1$ is a point-wise trivial cofibration (between diagrams with the same properties) and $\mu_2$ is a point-wise fibration and thus still point-wise in $\mathcal{W}_{loc}$.
Since the statement is clear for $\mu_1$ we may thus assume w.l.o.g.\@ that $\mu$ is a point-wise fibration. 

The morphism between pull-backs might be written as the following composition:
\[ X \times_Y Z \rightarrow X \times_{Y'} Z \rightarrow X \times_{Y'} Z' \rightarrow X' \times_{Y'} Z'.   \]
The last two morphisms are in $\mathcal{W}_{loc}$ because of 1.
The first is the following pullback of the diagonal $Y \rightarrow Y \times_{Y'} Y$:
\[ \xymatrix{ 
X \times_Y Z \ar[r] \ar[d]& Y \ar[d] \\
X \times_{Y'} Z \ar[r] & Y \times_{Y'} Y }  \]
in which the bottom morphism is a composition of pull-backs of morphisms in $\Fib$, hence in $\Fib$. 
Therefore it suffices to see that the diagonal $Y \rightarrow Y \times_{Y'} Y$ is in $\mathcal{W}_{loc}$. But that has a section $Y \times_{Y'} Y \rightarrow  Y$ which is a pullback of 
the morphism $Y \rightarrow Y'$ in $\Fib \cap \mathcal{W}_{loc}$ and thus it is in $\mathcal{W}_{loc}$ itself. This shows that fibrant replacement of diagrams in $\mathcal{M}$
also derives the pull-back functor in the Bousfield localization and hence the localization commutes with homotopy pull-back.  

$2. \Rightarrow 1.:$ Consider a Cartesian square
\[ \xymatrix{ 
X \times_Y Z \ar[r]^-{w'} \ar[d] & X \ar[d]^{f} \\
Z \ar[r]_-w & Y }  \]
in which $f \in \Fib$ and $w \in \mathcal{W}_{loc}$. Since $f$ is a fibration, the pull-back is a homotopy pull-back in $\mathcal{M}$ by right properness of $\mathcal{M}$. By assumption
it is also a homotopy pull-back in the localization. Since $w$ is a weak equivalence in the localization, also $w'$ must be. 
\end{proof}

\begin{SATZ}
A morphism of right derivators with domain homotopically finite diagrams commutes with pullbacks and terminal object if and only if it commutes with all homotopically finite limits. 
\end{SATZ}
\begin{proof}
This is (the dual of) \cite[Theorem~7.1]{PS14}.
\end{proof}

\section{Exact localizations of simplicial presheaf categories}

\begin{PAR}In their article \cite{DHI04} (cf.\@ in particular \cite[Proposition 5.1]{DHI04}) the authors investigate --- as a byproduct of their proofs --- to which extent weak equivalences in a left Bousfield localization of a model category of simplicial presheaves (the localization at all hypercovers) can
be characterized by local lifting properties along hypercovers. In the appendix we subsumed several general facts about the abstract notion of ``local lifting''. 
In this section and the next we show that weak equivalences in a left Bousfield localization of simplicial presheaves can be described by such a local lifting property precisely if the localization is exact. 
This works therefore also for the \v{C}ech local model structure. % which presents non-hypercomplete $\infty$-topoi. 
However, it is more of theoretical interest because the abstract lifting property is quite self-referential. 
\end{PAR}

\begin{PAR}
In this section we consider a model category structure $(\mathcal{SET}^{\mathcal{S}^{\op} \times \Delta^{\op}}, \Cof, \Fib, \mathcal{W})$ where $\mathcal{W}$ is the class of section-wise weak equivalences, and a left Bousfield localization $(\mathcal{SET}^{\mathcal{S}^{\op} \times \Delta^{\op}}, \Cof, \Fib_{loc}, \mathcal{W}_{loc})$ thereof.
{\em We assume that the first structure is simplicial, left and right proper, and cofibrantly generated, and that the class of cofibrations is contained in the class of monomorphisms. }
For instance, this holds for the projective or for the injective model category structure. 
\end{PAR}

\begin{PAR}
Recall that the injective structure is characterized by the fact that the cofibrations are the monomorphisms. Thus fibrations are those morphisms that have the right lifting property w.r.t.\@ all monomorphisms that are also section-wise weak equivalences. 
In the projective structure the fibrations are the section-wise surjective morphisms and the cofibrations are those morphisms $X \rightarrow Y$ for which the morphism\footnote{where $L_n$ denotes the $n$-th latching object w.r.t.\@ the Reedy structure on $\Delta^{\op}$}
\[ L_n Y \amalg_{L_n X} X_n \rightarrow Y_n \]
is of the form $A \rightarrow A \amalg \coprod B_i$, where the $B_i$ are retracts of representables. If $\mathcal{S}$ is idempotent complete (for example if it has fiber products) 
then the $B_i$ are representable themselves. There is an even more concrete description of the cofibrant objects (cf.\@ e.g.\@ \cite[Proposition 4.9]{Hor21}). In particular those are
degree-wise coproducts of retracts of representables.  
\end{PAR}

\begin{PAR}
 {\em We assume that also the Bousfield localization is cofibrantly generated or, equivalently, that it is a left  Bousfield localization generated by a set $S$ of cofibrations as in Theorem~\cite[A.3.7.3.]{Lur09}.} 
It follows that the set of weak equivalences $\mathcal{W}_{loc}$ is also part of a left Bousfield localization of the injective structure and we will sometimes use this fact. 
\end{PAR}

%Sometimes also the injective structure is used in some of the arguments. 
\begin{PAR}
The notions ``fibration'' and ``trivial cofibration'' will mean the corresponding notion for the global model structure. 
Because of the assumption on the existence of a subset $S \subset \Cof \cap \mathcal{W}_{loc}$ of generating cofibrations for the left Bousfield localization
 every trivial cofibration in the localization is a retract of a transfinite composition of trivial cofibrations and of push-outs of morphisms of the form
$(\partial \Delta_n \rightarrow \Delta_n) \boxplus f$, where $f$ is a cofibration in $S$. 
\end{PAR}

%Fix the class $\mathcal{COV}$ of morphisms $Y' \rightarrow Y$ which are in $\Fib \cap \mathcal{W}'$ and such that $Y'$ and $Y$ are cofibrant and such that there is $Y \rightarrow X$ in $\Fib \cap \mathcal{W}'$ such that $X$ is in $\mathcal{S}$.
%In the projective local model structure this is precisely the class of hypercovers (which are fibrations at the same time...). 
%In the projective Cech model structure this is the class of hypercovers which are also Cech local equivalences. They are not necessarity bounded.
%We do not know of a precise characterization which doesn't involve transfinite composition. 

\begin{BEISPIEL}\label{EX1}
Fix a Grothendieck pre-topology on $\mathcal{S}$ and
let $S$ be the set of (cofibration replacements of) the \v{C}ech covers. Those arise from a covering $\{U_i \rightarrow X\}$ of an object $X \in \mathcal{S}$ and are morphisms $U \rightarrow h_X$ where $U$ is the simplicial presheaf defined by
\[  U_n := \underbrace{ (\coprod_i h_{U_i}) \times_{h_X} \cdots  \times_{h_X} (\coprod_i h_{U_i}) }_{n-\text{times}}. \]
\end{BEISPIEL}

\begin{BEISPIEL}\label{EX2}
Fix a Grothendieck topology on $\mathcal{S}$ and let $S$ be the class of (cofibration replacements of) hypercovers. Hypercovers are morphisms of simplicial presheaves 
\[ Y \rightarrow h_X \]
in which $X \in \mathcal{S}$ and $Y$ is degree-wise a coproduct of representables such that the morphism
\[ Y_n \rightarrow \Hom(\partial \Delta_n, Y) \times_{\Hom(\partial \Delta_n, h_X)} h_X  \]
is a local epimorphism for any $n$.
The class $S$ might not be a set. In \cite[6.5]{DHI04} it is shown that one replace $S$ by a {\em dense} set $S' \subset S$ of hypercovers without changing the localization. In particular
 {\em all} hypercovers are still weak equivalences in the localization. 
\end{BEISPIEL}

\begin{BEM}The above definition of hypercover is a special case of the following more general definition. A (generalized) hypercover between two simplicial presheaves is
a morphism
\[ Y \rightarrow X \]
such that 
\[ Y_n \rightarrow \Hom(\partial \Delta_n, Y) \times_{\Hom(\partial \Delta_n, X)} X  \]
is a {\em local epimorphism} for any $n$. For a general morphism $A \rightarrow B$ of presheaves this means that for any $X \in \mathcal{S}$ and section $h_X \rightarrow B$ 
there exists a covering $\{U_i \rightarrow X\}$ and an extension to a commutative square as follows:
\[ \xymatrix{ \coprod_{i} h_{U_i}  \ar@{.>}[d] \ar@{.>}[r] & A \ar[d] \\ 
h_X \ar[r] & B 
} \]
\end{BEM}
We have the following formal property: 
\begin{LEMMA}\label{LEMMAHCPULLBACK}
 Generalized hypercovers are closed under pull-back.
\end{LEMMA}
\begin{proof}
By definition a generalized hypercover $X \rightarrow Y$ has the property that 
\[ X_n \rightarrow Y_n \times_{\Hom(\partial \Delta_n, Y)} \Hom(\partial \Delta_n, X)  \]
are local epimorphisms. For a pull-back diagram
\[ \xymatrix{
X' \ar[r]^{}  \ar[d]_{f'} & X \ar[d]^f \\
Y' \ar[r]^{} & Y
}\]
consider the diagram 
\[ \xymatrix{
X'_n  \ar[d]\ar[rr] &&  X_n \ar[d] \\
\Box \ar[rd] \ar[rr]^{}  \ar[dd] & &\Box\ar[rd] \ar[dd] \\
& \Hom(\partial \Delta_n, X') \ar[rr]^{}  \ar[dd] & & \Hom(\partial \Delta_n, X) \ar[dd] \\
Y'_n \ar[rd] \ar[rr]^{} & & Y_n \ar[rd] \\
& \Hom(\partial \Delta_n, Y') \ar[rr]^{} & & \Hom(\partial \Delta_n, Y) 
}\]
in which the front and back (total) square are Cartesian and also (by definition) the right and left squares. It follows that the back lower square is Cartesian and 
hence also the back upper square is Cartesian. Therefore, since the upper right vertical morphism is a local epimorphism so is the left. 
\end{proof}

The goal of this section and the next is to establish several equivalent conditions for a given localization $\mathcal{SET}^{\mathcal{S}^{\op} \times \Delta^{\op}}_{loc}$ to be exact in the sense of Theorem~\ref{SATZPROPEREXACTLOC}. The first is: 

%Fix a right proper model category structure on simplicial presheaves such that all cofibrations are monomorphisms. 
%Let $S$ be a class of morphisms of simplicial presheaves and let $\mathcal{W}'$ be the saturation defined by means of the left Bousfield localization. 

\begin{PROP}\label{PROPEXACT}
If each $f \in S$ has property $\mathbf{P}_{fib}$ then the whole class $\mathcal{W}_{loc}$ has property $\mathbf{P}_{fib}$. 
\end{PROP}

By Theorem~\ref{SATZPROPEREXACTLOC} this is equivalent to exactness and hence will hold for (the localization of) any right proper model category structure with the same weak equivalences. %and thus also e.g.\@ the projective structure and its left Bousfield localization (which has the same weak equivalences). 

We need a couple of lemmas: 

\begin{LEMMA}\label{LEMMAP2}
\begin{enumerate}
\item If $I$ is an ordinal and $F: I \rightarrow \mathcal{C}$ a functor mapping each morphism (i.e.\@ relation) to a cofibration with property $\mathbf{P}_{fib}$ then also the morphism
$F(0) \rightarrow \colim F$ (the transfinite composition) has property $\mathbf{P}_{fib}$.
\item Consider a push-out diagram
\[ \xymatrix{  
X \ar[d]_f \ar[r]^g & X' \ar[d]^F \\
Y \ar[r]_G & Y'  }\]
in which $f$ or $g$ is a monomorphism. If $f$ has property $\mathbf{P}_{fib}$  also $F$ has property $\mathbf{P}_{fib}$. 
\item If $f$ is a retract of $g$ and $g$ has property $\mathbf{P}_{fib}$ then $f$ has property $\mathbf{P}_{fib}$.
\end{enumerate}
\begin{proof}
1. Consider a morphism $X \rightarrow \colim f$ and the corresponding constant diagrams $(X), (\colim f): I \rightarrow \mathcal{C}$. Consider the pull-back
\[ (X) \times_{(\colim F)} F. \]
It is still a diagram of monomorphisms in $\mathcal{W}_{loc}$. Hence
\[ X \times_{\colim F} F(0) \rightarrow \colim ((X) \times_{(\colim F)} F)  \]
is in $\mathcal{W}_{loc}$ (using the injective structure in which the monomorphisms are the cofibrations).
Since pull-back commutes with filtered colimits this morphism is the same as 
\[ X \times_{\colim F} F(0) \rightarrow X. \]
The transfinite composition has thus property $\mathbf{P}_{fib}$.

2. We may argue as in 1.\@ using that push-outs along monomorphisms commutes with fiber products.

4. We may argue as in 1.\@ using that fiber products commute with retracts.  
\end{proof}

\end{LEMMA}

\begin{LEMMA}\label{LEMMABOXPLUSPFIB}
If $f: X \rightarrow Y$ has property $\mathbf{P}_{fib}$ then also 
\[  (\partial \Delta_n \rightarrow \Delta_n) \boxplus f \] 
has property $\mathbf{P}_{fib}$. 
\end{LEMMA}
\begin{proof}
Consider the  diagram
\[ \xymatrix{
\partial \Delta_n \otimes X \ar[r] \ar[d] & \Delta_n \otimes X \ar[d] \\
\partial \Delta_n \otimes Y \ar[r] & \lefthalfcap \ar[d] \\
& \Delta_n \otimes Y 
}\]
Since $\partial \Delta_n \otimes X \rightarrow \Delta_n \otimes X$ is a monomorphism, by Lemma~\ref{LEMMAP1}, 2.\@ and Lemma~\ref{LEMMAP2}, 2.\@ it suffices to show that  $\partial \Delta_n \otimes X \rightarrow \partial \Delta_n \otimes Y$ and $\Delta_n \otimes X \rightarrow \Delta_n \otimes Y$ have property $\mathbf{P}_{fib}$.
Factor $f : X \rightarrow Y$ as 
\[ \xymatrix{  X \ar[r]^-g & X' \ar[r]^-h & Y } \] 
with $g \in \Cof \cap \mathcal{W}$ and $h \in \Fib$. By 2-out-of-3 we have $h \in \mathcal{W}_{loc}$. 
In the following diagram all squares are Cartesian and $K$ is an arbitary simplicial set:
\[ \xymatrix{
 \Box \ar[r] \ar[d]_{\mathcal{W}} &  K \otimes X \ar[r] \ar[d]^{\mathcal{W}} & X \ar[d]^{g \in \Cof \cap \mathcal{W}} \\
 \Box \ar[r]^-{\Fib} \ar[d]_{\mathcal{W}_{loc}} &  K \otimes X' \ar[r] \ar[d] & X' \ar[d]^{h \in \Fib \cap \mathcal{W}_{loc}} \\
Z \ar[r]_-{\Fib} & K \otimes Y \ar[r] & Y
}\]
The lower left vertical morphism is in $\mathcal{W}_{loc}$ because by Lemma~\ref{LEMMAP1}, 1.\@ $h: X' \rightarrow Y$ has property $\mathbf{P}$ and the upper left vertical morphism is in $\mathcal{W}_{loc}$ because of right properness and the fact that the middle top vertical morphism is in $\mathcal{W}$.
\end{proof}

%Now, if we would like to see that the cofibrant replacement of a Cech cover has the pull-back property we may restrict to cofibrant objects and also
%to the Cech cover itself... ...

%Let $f: C \rightarrow X$ be the Cech cover and let $C \rightarrow X' \rightarrow X$ be a factorization into cofibration followed by trivial fibration.
%Then, if any pull-back of $f$ along a fibration is in $\mathcal{W}'$ then also any pull-back of $C \rightarrow X'$ along a fibration is in $\mathcal{W}'$.

\begin{LEMMA}\label{LEMMAPHOM}
Consider two morphisms $f, g: X \rightarrow Y$ which are left homotopic. 
Then $f$ has property $\mathbf{P}_{fib}$ if and only if $g$ has property $\mathbf{P}_{fib}$.
\end{LEMMA}
\begin{proof}Let $\mu: \Delta_1 \otimes X \rightarrow Y$ be the homotopy between $f$ and $g$. 
Consider for $i \in \{0,1\}$ the pull-back diagram along a fibration:
\[ \xymatrix{
Z_i  \ar[r] \ar[d]_{\mathcal{W}} & \{ i \}\otimes X \ar[d]^{e_i \in \mathcal{W}} \ar[d]  \\
Z' \ar[r] \ar[d] & \Delta_1 \otimes X \ar[d]^\mu \\
Z \ar[r]_-{\Fib} & Y
}\]
Note that $e_i: \{i\} \otimes X \rightarrow \Delta_1 \otimes X$ is in $\mathcal{W}$ (existence of the injective structure which is simplicial and in which $X$ is cofibrant). 
Hence the upper left vertical morphism is in $\mathcal{W}$ because of right properness. 
Assuming $Z_0 \rightarrow Z$ is in $\mathcal{W}_{loc}$,
the first pull-back shows therefore that $Z' \rightarrow Z$ is in $\mathcal{W}_{loc}$ and so is $Z_1 \rightarrow Z$.
\end{proof}

\begin{proof}[{Proof of Proposition~\ref{PROPEXACT}}]
Let $f \in \mathcal{W}_{loc}$. It is the composition of a cofibration in $\mathcal{W}_{loc}$ and a trivial fibration. Since trivial fibrations are closed under pull-back we may assume w.l.o.g.\@ that $f$ is a cofibration. Then $f$ is a retract of a transfinite composition of pushouts of trivial cofibrations or morphisms of the form 
\[  (\partial \Delta_n \rightarrow \Delta_n) \boxplus f \]
in which $f$ is a cofibration in $S$. The morphism $f$ has property $\mathbf{P}_{fib}$ by assumption. % and Lemma~\ref{LEMMAPFIBCOFIBRANT1}.
Trivial cofibrations have property $\mathbf{P}_{fib}$ because the model category is right proper. 
We conclude by Lemma~\ref{LEMMAP2} and Lemma~\ref{LEMMABOXPLUSPFIB}.
\end{proof}

\begin{LEMMA}\label{LEMMAPFIBCOFIBRANT1}
Let $f: X \rightarrow Y$ be a morphism with property $\mathbf{P}_{fib}$ between cofibrant objects. Factor $f$ 
as
\[ \xymatrix{ X \ar[r]^g & Y' \ar[r]^h & Y } \]
with $g \in \Cof$ and $h \in \Fib \cap \mathcal{W}$. Then $g$ has property $\mathbf{P}_{fib}$ .  
\end{LEMMA}
\begin{proof}
Since $Y'$ is cofibrant  there is a section $\sigma: Y \rightarrow Y'$ such that the composition $\sigma h: Y' \rightarrow Y'$ is left homotopic to the identity. Therefore the  morphisms $g, \sigma f: X \rightarrow Y'$ are left homotopic as well. Since $f$ and $\sigma$ have property $\mathbf{P}_{fib}$, by Lemma~\ref{LEMMAPHOM}, the same holds for $g$.
\end{proof}

\begin{LEMMA}\label{LEMMAPFIBCOFIBRANT2}
If a morphism $f$ has the property that pull-backs along fibrations (resp.\@ morphisms) {\em with cofibrant source} are in $\mathcal{W}_{loc}$ then $f$ has property $\mathbf{P}_{fib}$ (resp.\@ property $\mathbf{P}$). 
\end{LEMMA}
\begin{proof}
Consider a diagram with Cartesian squares in which $W$ is a cofibrant replacement of $Z$: 
\[ \xymatrix{
W' \ar[rr]^{\Fib \cap \mathcal{W}}  \ar[d] & & Z' \ar[r]  \ar[d] & X \ar[d] \\
W \ar[rr]_{\Fib \cap \mathcal{W}}  & & Z \ar[r] & Y 
}\]
By assumption $W' \rightarrow W$ is in $\mathcal{W}_{loc}$. Thus by 2-out-of-3 the same holds for $Z' \rightarrow Z$ . 
\end{proof}

%WHAT IS INVERTED IN THE HYPERCOVER COMPLETION? REL.\@ HYPERCOVERS OR HYPERCOVERS OF OBJECT IN S ONLY?,

\begin{LEMMA}\label{LEMMABISIMP}
If $X \rightarrow Y$ is a morphism of bisimplicial presheaves in $\mathcal{SET}^{\mathcal{S}^{\op}\times \Delta^{\op}\times \Delta^{\op}}$
such that the horizontal morphisms of simplicial presheaves
\[ X_{\bullet,i} \rightarrow Y_{\bullet,i} \]
are in $\mathcal{W}_{loc}$ then the map of diagonal simplicial presheaves
\[ \delta^*X \rightarrow \delta^*Y \]
is in $\mathcal{W}_{loc}$. Here $\delta:  \Delta^{\op} \rightarrow \Delta^{\op}\times \Delta^{\op}$ is the diagonal. 
\end{LEMMA}
\begin{proof}
This follows from the fact that the homotopy colimit over $\Delta^{\op}$ can be computed by
the diagonal of a horizontally point-wise cofibrant diagram (hence no restriction in the injective model structure). Since the localization commutes with homotopy colimits we conclude. 
\end{proof}

\begin{PROP}\label{PROPCECHANDHYPERCOVERS}
\v{C}ech covers  (cf.\@ \ref{EX1}) of an object $X \in \mathcal{S}$ and hypercovers (cf.\@ \ref{EX2}) have property $\mathbf{P}$ w.r.t.\@ the respective class $\mathcal{W}_{loc}$.
%If $f: U \rightarrow X$ is a Cech cover then any pull-back of $f$ is a weak equivalence, i.e.\@ it has property $P$.
\end{PROP}
\begin{proof}
We make two preliminary considerations:

1.\@ If 
$U \rightarrow h_X$
is a \v{C}ech cover and
$Z \rightarrow h_X$
a morphism in which $Z$ is a coproduct of retracts of representables then also the pull-back
\[ U \times_{h_X} Z \rightarrow Z \]
is in $\mathcal{W}_{loc}$. For, if $Z = h_{X'}$ is representable, then $U \times_{h_X} h_{X'}$ is (by definition of a Grothendieck pre-topology) degree-wise representable and the morphism 
is a \v{C}ech-cover again. Since $\mathcal{W}_{loc}$ is closed under coproducts\footnote{because of the existence of the injective structure in which all objects are cofibrant}
this reduces the claim to the case in which $Z$ is a retract of a representable:
\[ Z \rightarrow h_{X'} \rightarrow Z \]
But then $U \times_{h_X} Z \rightarrow Z$ is a retract of $U \times_{h_X} h_{X'} \rightarrow h_{X'}$ and thus in $\mathcal{W}_{loc}$ as well.

2.\@ If $Y \rightarrow h_X$ is a hypercover and $Z \rightarrow h_X$
is a morphism in which $Z$ is a coproduct of retracts of representables, we claim that 
the pull-back
\[ Y \times_{h_X} Z \rightarrow Z \]
is again in $\mathcal{W}_{loc}$. Like before this is reduced to the case in which $Z = h_{X'}$ is representable. 
By Lemma~\ref{LEMMAHCPULLBACK}
the morphism 
\[ Y \times_{h_X} h_{X'} \rightarrow h_{X'} \]
is a generalized hypercover. Choose a projectively cofibrant replacement (this may be taken to be degree-wise a coproduct of representables) 
\[ Y' \rightarrow Y \times_{h_X} h_{X'} \rightarrow h_{X'}. \]
Then $Y' \rightarrow Y \times_{h_X} h_{X'}$, being a projective fibration, is also a generalized hypercover (in which the local epimorphisms are even split epimorphisms). 
Therefore $Y' \rightarrow h_{X'}$ is in $\mathcal{W}_{loc}$ and so is $Y \times_{h_X} h_{X'} \rightarrow h_{X'}$. 

Now we are able to prove the statement. Consider a pull-back diagram
\[ \xymatrix{
\Box \ar[r]^{}  \ar[d] & U \ar[d] \\
A \ar[r]^{} & h_X
}\]
in which $U \rightarrow h_X$ is a \v{C}ech cover, or an hypercover, respectively. 
We may assume that $A$ is cofibrant w.r.t.\@ the projective model category structure (Lemma~\ref{LEMMAPFIBCOFIBRANT2}), in particular is degree-wise a coproduct of retracts of representables. Then form the bisimplicial set 
\[ X_{i,j} := U_i \times_ X A_i. \]
The pullback is the diagonal of this bisimplicial set.
We have a morphism
\[ X \rightarrow A  \]
where $A$ is seen as a bisimplicial presheaf constant in the horizontal direction. 
Therefore the map on the diagonals (i.e.\@ homotopy colimits) is in $\mathcal{W}_{loc}$ by Lemma~\ref{LEMMABISIMP} if each
horizontal morphism
\[ X_{i,\bullet} \rightarrow A_i \]
is in $\mathcal{W}_{loc}$. This has been shown in the preliminary considerations. 
\end{proof}

We obtain the well-known fact: 

\begin{KOR}\label{COREXACT}
Let $\mathcal{S}$ be a small category with Grothendieck (pre-)topology. 
The left Bousfield localizations of $\mathcal{SET}^{\mathcal{S}^{\op} \times \Delta^{\op}}$ at all \v{C}ech covers (cf.\@ \ref{EX1}), and at all hypercovers  (cf.\@ \ref{EX2}), respectively, are exact in the sense of Theorem~\ref{SATZPROPEREXACTLOC}.
\end{KOR}
\begin{proof}
The statement is independent of the choice of model category structure (satisfying the assumptions in the beginning of this section). 
We may thus work with the injective model category structure. Let $S$ be a class of cofibration replacements of the \v{C}ech covers (resp.\@ hypercovers). 
Proposition~\ref{PROPCECHANDHYPERCOVERS} and Lemma~\ref{LEMMAPFIBCOFIBRANT1} show that the morphisms in $S$ have property $\mathbf{P}_{fib}$.
We conclude by Proposition~\ref{PROPEXACT}.
\end{proof}

\section{Coverings of simplicial presheaves}

\begin{PAR}
In this section we fix the {\em projective} model category structure on $\mathcal{SET}^{\mathcal{S}^{\op} \times \Delta^{\op}}$.
\end{PAR}

%Fix a fibrant replacement functor $Q$. Define a class $\mathcal{W}''$ of morphisms $f$ such that $Q(f)$ has the homotopy lifting property.

\begin{PAR}\label{PARAXIOMS}
Fix a subcategory $\mathcal{COV} \subset \mathcal{SET}^{\mathcal{S}^{\op} \times \Delta^{\op}}$ of ``coverings''. Consider the following axioms on $\mathcal{COV}$:
\begin{enumerate}
\item[(C1)] Each object of $\mathcal{COV}$ is cofibrant.
\item[(C2)] Every representable presheaf (considered as constant simplicial presheaf) is in $\mathcal{COV}$.
%\item[(C3)] Trivial fibrations are in $\mathcal{COV}$.

\item[(C3)] Each morphism of $\mathcal{COV}$ is in $\mathcal{W}_{loc}$.
\item[(C4)] If $Y \in \mathcal{COV}$ and $Y' \rightarrow Y$ is in $\Fib \cap \mathcal{W}_{loc}$ with $Y'$ cofibrant then $Y' \rightarrow Y$ is in $\mathcal{COV}$.
\end{enumerate}

There is an obvious smallest and biggest choice of $\mathcal{COV}$ that satisfies (C1--4), but the choice will not matter. 
\end{PAR}

We now define a class of ``local weak equivalences'' depending on $\mathcal{COV}$ as follows:
\begin{DEF}\label{DEFWCOV}
Let $\mathcal{W}_{\mathcal{COV}}$ be the class of morphisms $f$ for which there is a diagram
\[ \xymatrix{
X \ar[r]^{\mathcal{W}}  \ar[d]_f & X' \ar[d]^{f'} \\
Y \ar[r]_{\mathcal{W}} & Y'
}\]
in which $f'$ has the local homotopy lifting property w.r.t.\@ $\mathcal{COV}$ (cf.\@ Definition~\ref{DEFHLP}), $X'$ and $Y'$ are 
%locally fibrant (w.r.t.\@ $\mathcal{COV}$, cf.\@ Definition~\ref{DEFLOCFIB}) 
fibrant and the horizontal morphisms are in $\mathcal{W}$. 
\end{DEF}

\begin{LEMMA}\label{LEMMAWCOV}
Let $\mathcal{COV}$ satisfy (C1) and (C2) but not necessarily the other axioms. 
\begin{enumerate}
\item Definition~\ref{DEFWCOV} is independent of the fibrant replacement. 
\item $\mathcal{W}_{\mathcal{COV}}$ satisfies 2-out-of-3. 
\item The class $\mathcal{W}_{\mathcal{COV}}$ is stable under pull-back along fibrations. 
\item We have $\mathcal{W} \subset \mathcal{W}_{\mathcal{COV}}$. 
\end{enumerate}
\end{LEMMA}
\begin{proof}
1. Consider a refinement of fibrant replacements, i.e.\@ a diagram
\[ \xymatrix{ X \ar[r]^{\mathcal{W}}  \ar[d]^f & X' \ar[r]^{\mathcal{W}}  \ar[d]^{f'} & X'' \ar[d]^{f''} \\
Y \ar[r]_{\mathcal{W}} & Y' \ar[r]_{\mathcal{W}} & Y''
} \]
in which $X', Y', X''$ and $Y''$ are fibrant and the horizontal morphisms are in $\mathcal{W}$. A morphism in $\mathcal{W}$ between fibrant 
objects has the homotopy lifting property. 
Hence $f'$ has the homotopy local lifting property if and only if $f''$ has the homotopy local lifting property by 2-out-of-3 (Lemma~\ref{LEMMA2OUTOF3}).
Since any two fibrant replacements can be refined by a third one the statement follows. 

2. follows from 1.\@ using any functorial fibrant replacement and 2-out-of-3 for homotopy local lifting properties (Lemma~\ref{LEMMA2OUTOF3}).

3. Let $f: X \rightarrow Y$ be a morphism in $\mathcal{W}_{\mathcal{COV}}$. We can find a diagram
\[ \xymatrix{ X \ar[r]^{\mathcal{W}} \ar[d] & X' \ar[d]^{\Fib}  \\
Y \ar[r]_{\mathcal{W}} & Y' } \]
in which $X'$ and $Y'$ are fibrant. Let $Z \rightarrow Y$ be a given fibration. 
 We get the diagram
\[ \xymatrix{
X \times_Y Z \ar[rr] \ar[dd]_{\mathcal{W}} & & X \ar[dd]^(.3){\mathcal{W}} \\
& X' \times_{Y'} Z' \ar[rr]  \ar[dd]^(.3){\Fib} & & X' \ar[dd]^{\Fib} \\
\Box \ar[rr]_(.3){\Fib} \ar[ru]^{\mathcal{W}}  \ar[dd]_{\Fib} & & \Box \ar[ru]^{\mathcal{W}} \ar[dd]^(.2){\Fib} \\
& Z' \ar[rr] & & Y' \\
Z \ar[ru]^{\mathcal{W}} \ar[rr]_{\Fib} & & Y \ar[ru]_{\mathcal{W}}
} \]
in which $Z \rightarrow Z' \rightarrow Y'$ is the factorization of the  composition $Z \rightarrow Y \rightarrow Y'$ into trivial cofibration followed by fibration and
in which the middle floor is the pullback of the bottom floor under the fibration $X' \rightarrow Y'$. Hence all 5 upright squares are Cartesian. 
Thus the so indicated morphisms are weak equivalences by right properness. 
Note that  $X' \times_{Y'} Z' \rightarrow Z'$ has the local lifting property (Lemma~\ref{LEMMAPULLBACK}). 
This shows that $X \times_Y Z \rightarrow Z$ has a fibrant replacement which has the local lifting property. It is thus in $\mathcal{W}_{\mathcal{COV}}$ as well.

4. Let $f: X \rightarrow Y$ a morphism in $\mathcal{W}$. As in 3.\@ we may replace $f$ by a fibration which is then trivial by 2-out-of-3. 
A trivial fibration has obviously the (local) lifting property. 
\end{proof}

\begin{PROP}\label{PROPSELFLIFTING1}
If $\mathcal{COV}$ satisfies (C1) and (C2) then 
a fibration in $\mathcal{W}_{\mathcal{COV}}$ has the local lifting property itself. 
\end{PROP}
\begin{proof}
Let $f \in \mathcal{W}_{\mathcal{COV}} \cap \Fib$. We find a commutative diagram
\[ \xymatrix{ 
  X \ar[rr]^{\Cof \cap \mathcal{W}} \ar[d]_{f} &&   \ar[d]^{f'} X' \\
  Y \ar[rr]_{\Cof \cap \mathcal{W}} & & Y'
} \]
in which $f'$ is a fibration between fibrant objects that has the local lifting property. 
Now form the pull-back and factor the induced morphism as indicated: 
\[ \xymatrix{ 
X \ar[rd]^{\Cof \cap \mathcal{W}} \ar@/^20pt/[rrrrdd] \ar@/_20pt/[rrdddd]_f  \\
&  \ar[rd]^{h \in \Fib \cap \mathcal{W}} &  & \\
& &  \Box \ar[dd]^{f''} \ar[rr]^{\mathcal{W}}  & & X'  \ar[dd]^{f'} \\
\\
& & Y \ar[rr]_{\Cof \cap \mathcal{W}} & & Y'
} \]
 
Here $f''$ has the local lifting property. Thus $f'' h$ has the local lifting property and thus also $f$ by Lemma~\ref{LEMMATFIBLIFTING}. 
\end{proof}

%Define $\mathcal{W}'' := \mathcal{W}_{\mathcal{COV}}$, where $\mathcal{COV}$ is the class of all fibrations $X \rightarrow X'$ between cofibrant objects which are in $\mathcal{W}'$ and such that there is a morphism $X \rightarrow X' \rightarrow Y$ where $Y$ is in $\mathcal{S}$.

%Refinements: Actually we can replace $\mathcal{COV}$ here by all $\mathcal{W}'$-weak equivalences between cofibrant objects or local fibrations between cofibrant objects and so on
%because the Key Lemma below is equally true. 

$\mathcal{W}_{\mathcal{COV}}$ shares the following property with $\mathcal{W}_{loc}$:
\begin{LEMMA}\label{LEMMAOFIBTFIB}Assume that $\mathcal{COV}$ satisfies (C1), (C2), and (C3), then 
\[ \Fib_{loc} \cap \mathcal{W}_{\mathcal{COV}} \subset \Fib \cap \mathcal{W}. \] 
\end{LEMMA}
\begin{proof}
Let $f \in \Fib_{loc} \cap \mathcal{W}_{\mathcal{COV}}$. 
Consider a diagram
\[ \xymatrix{
  \partial \Delta_n \otimes h_X \ar[r]^{}  \ar[d] & A \ar[d]^f \\
 \Delta_n  \otimes h_X  \ar[r]^{} & B 
}\]
in which $X$ is an object in $\mathcal{S}$ (and thus $h_X \in \mathcal{COV}$ by C2). Since $f$ is in particular in $\Fib$, by Proposition~\ref{PROPSELFLIFTING1}, it has itself the local lifting property. Thus
there is a lift
\[ \xymatrix{
\partial \Delta_n \otimes X'   \ar[r]^{}  \ar[d] & \partial \Delta_n  \otimes h_X  \ar[r]^{}  \ar[d] & A \ar[d]^f \\
\Delta_n \otimes X'    \ar@{-->}[rru]^h \ar[r]^{} & \Delta_n \otimes h_X   \ar[r]^{} & B 
}\]
in which $X' \rightarrow h_X$ is in $\mathcal{COV}$ and hence in $\mathcal{W}_{loc}$ by (C3). Factor $X' \rightarrow h_X$ as 
\[ \xymatrix{X' \ar[rr]^{\Cof \cap \mathcal{W}_{loc}} & & X'' \ar[rr]^{\Fib \cap \mathcal{W}} & & h_X. } \]
Then $h$ descends by definition of $\Fib_{loc}$ because 
the morphism 
\[ (\partial \Delta_n \to \Delta_n) \boxplus  (X' \rightarrow X'') \]
is a cofibration in $\mathcal{W}_{loc}$:

\[ \xymatrix{
\partial \Delta_n \otimes X'   \ar[r]^{}  \ar[d] & \partial \Delta_n\otimes X''   \ar[r]^{}  \ar[d] &  \partial \Delta_n \otimes h_X \ar[r]^{}  \ar[d] & A \ar[d]^f \\
\Delta_n \otimes X'    \ar[r]^{} & \Delta_n\otimes X''  \ar@{-->}[rru]^h \ar[r]^{} & \Delta_n \otimes h_X   \ar[r]^{} & B 
}\]
The morphisms $X'' \rightarrow h_X$ are trivial fibrations between cofibrant objects and thus are deformation retractions.
This shows that we have finally also a lift
\[ \xymatrix{
\partial \Delta_n \otimes h_X   \ar[r]^{}  \ar[d] & A \ar[d]^f \\
\Delta_n \otimes h_X    \ar[r]^{} \ar@{-->}[ru] & B 
}\]
Hence $f$ is a trivial projective fibration. 
\end{proof}

\begin{LEMMA}\label{THELEMMA}
Assume that $\mathcal{COV}$ satisfies (C1), (C2), and (C3). Then
$\mathcal{W}_{loc} \subset \mathcal{W}_{\mathcal{COV}}$ implies $\mathcal{W}_{loc} = \mathcal{W}_{\mathcal{COV}}$.
\end{LEMMA}
\begin{proof}
This is a standard argument using Lemma~\ref{LEMMAOFIBTFIB}. Let $f \in \mathcal{W}_{\mathcal{COV}}$. Factor $f$ as cofibration followed by a trivial fibration. The trivial fibration is in $\mathcal{W}_{\mathcal{COV}}$. Thus it suffices to show:
Given a fibration $g \in \Fib_{loc}$ and $f \in  \mathcal{W}_{\mathcal{COV}} \cap \Cof$ then $f$ has the left lifting property w.r.t.\@ $g$. Consider the following commutative diagram in which the right hand square is Cartesian and factor $X \rightarrow \Box$ as indicated:
\[ \xymatrix{
X \ar[dd]_f \ar[rrrr] \ar[rrd]^{\mathcal{W}_{loc} \subset \mathcal{W}_{\mathcal{COV}}} && && \Box \ar[r] \ar[dd]^{\Fib_{loc}} & A \ar[dd]^{g\in \Fib_{loc}} \\
&& X' \ar[rru]_{\Fib_{loc}}  \ar[rrd]^{\Fib_{loc}} \\
Y \ar@{-->}[rru]\ar@{=}[rrrr]&&&& Y \ar[r] & B 
} \]
Since the morphism $X \rightarrow Y$ is in $\mathcal{W}_\mathcal{COV}$ also $X' \rightarrow Y$ must be in $\mathcal{W}_\mathcal{COV}$ and hence by Lemma~\ref{LEMMAOFIBTFIB} in $\Fib \cap \mathcal{W}$. Therefore a lift indicated by the dotted arrow exists using merely that $f: X \rightarrow Y$ is a cofibration. 
\end{proof}

\begin{PROP}\label{PROPSELFLIFTING2}Assume that $\mathcal{COV}$ satisfies (C1), (C2), and (C3).
If $\mathcal{W}_{loc} \subset \mathcal{W}_{\mathcal{COV}}$ then a fibration is in $\mathcal{W}_{\mathcal{COV}}$ if and only if it has the local lifting property itself. 
\end{PROP}
\begin{proof}
One direction is Proposition~\ref{PROPSELFLIFTING1}.
For the converse assume that $f$ is a fibration which has the local lifting property.
Factor $f$ as a morphism in $\Cof \cap \mathcal{W}_{loc}$ followed by a morphism in $\Fib_{loc}$. Then factor the morphism in $\Cof \cap \mathcal{W}_{loc}$ as trivial cofibration followed by
fibration. The fibration is in $\mathcal{W}_{loc}$ by 2-out-of-3 and thus in $\mathcal{W}_{\mathcal{COV}}$ by assumption:
\[ \xymatrix{ X \ar[rr]^{\Cof \cap \mathcal{W}}&  & X' \ar[rr]^-{\Fib \cap \mathcal{W}_{\mathcal{COV}}} & & X'' \ar[rr]^{\Fib_{loc}}  & & Y.  } \]
By Lemma~\ref{LEMMATFIBLIFTING} the fibration (composition of the last 2 morphisms) has the local lifting property. The second morphism has also the local lifting property by Proposition~\ref{PROPSELFLIFTING1}. 
Therefore, by 2-out-of-3 (Lemma~\ref{LEMMA2OUTOF3}), also the morphism in $\Fib_{loc}$ has the local lifting property and is therefore by the proof of Lemma~\ref{LEMMAOFIBTFIB} a trivial fibration.
In total the morphism $f$ is in $\mathcal{W}_{\mathcal{COV}}$.
\end{proof}

The following theorem summarizes the discussion:

\begin{SATZ}\label{MAINTHEOREMEXACT}
Let $\mathcal{S}$ be a small category. 
Choose the projective model category structure on $\mathcal{SET}^{\mathcal{S}^{\op} \times \Delta^{\op}}$. Consider a (cofibrantly generated) left Bousfield localization with class $\mathcal{W}_{loc}$ of weak equivalences. 
Let $\mathcal{COV}$ be a subcategory of coverings satisfying (C1--4) of \ref{PARAXIOMS}.
The following are equivalent:
\begin{enumerate}
\item $\mathcal{W}_{loc}$ is stable under pull-back along fibrations;
\item $S$, a generating set of cofibrations, goes to $\mathcal{W}_{loc}$ under pull-back along fibrations with cofibrant source;
\item The left Bousfield localization is exact, i.e.\@ the localization functor (left adjoint) commutes with (homotopically) finite homotopy limits;
\item $\mathcal{W}_{loc} \subset \mathcal{W}_{\mathcal{COV}}$;
\item $\mathcal{W}_{loc} = \mathcal{W}_{\mathcal{COV}}$.
\end{enumerate}
\end{SATZ}

\begin{proof}
The implication $2. \Rightarrow 1.$ is Proposition~\ref{PROPEXACT} and $1. \Rightarrow 2.$ is clear. 

The equivalence $1. \Leftrightarrow 3.$ is Theorem~\ref{SATZPROPEREXACTLOC}.

$5. \Rightarrow 1.$ By Lemma~\ref{LEMMAWCOV}, 3.\@ $\mathcal{W}_{\mathcal{COV}}$ is stable under pull-back along fibrations. 

$1. \Rightarrow 4.$
Consider $Y \in \mathcal{COV}$ and $f \in \mathcal{W}_{loc}$. We may choose a fibrant replacement $f'$ of $f$ which is also a fibration.
Hence we have
a commutative square
\[ \xymatrix{
\partial \Delta_n \otimes Y   \ar[r]^{}  \ar[d] & A \ar[d]^{f'} \\
\Delta_n \otimes Y    \ar[r]^{} & B 
}\]
in which $f' \in \Fib \cap \mathcal{W}_{loc}$ and $A$ and $B$ are fibrant objects. We have to show that a local lifting exists. 
Consider the diagram with Cartesian squares
\[ \xymatrix{
Y' \ar[d]_{ \mathcal{W} \cap \Fib } \\ 
\Box \ar[r]  \ar[d]_{\mathcal{W}_{loc} \cap \Fib} & \Hom(\Delta_n, A) \ar[d] \\
Y \ar[r] & \Box \ar[r]^{}  \ar[d] & \Hom(\partial \Delta_n, A) \ar[d] \\
& \Hom( \Delta_n, B)  \ar[r]^{} &  \Hom(\partial \Delta_n, B)
}\]
in which the indicated vertical morphism is in $\mathcal{W}_{loc}$ by Lemma~\ref{LEMMABOXPLUSPFIB} and Lemma~\ref{LEMMAP1}, 1., and $Y' \rightarrow \Box$ is a cofibrant replacement. 
Putting things together we get a lift in
 \[ \xymatrix{
\partial \Delta_n \otimes Y'   \ar[r]^{} \ar[d] &  \partial \Delta_n \otimes Y   \ar[r]^{}  \ar[d] & A \ar[d]^f \\
\Delta_n \otimes Y'    \ar[r]^{} \ar@{.>}[rru]^{} &   \Delta_n \otimes Y \ar[r]^{} & B 
}\]
and the morphism $Y' \rightarrow Y$ is in $\mathcal{COV}$ by (C4).

$4. \Rightarrow 5.$ is Lemma~\ref{THELEMMA}.
\end{proof}

\begin{PAR}
Consider the case of the left Bousfield localization at the \v{C}ech covers  (cf.\@ \ref{EX1}). We may define the smallest class $\mathcal{COV}$ that contains
all representable constant presheaves and inductively all morphisms in $\Fib \cap \mathcal{W}_{loc}$ between cofibrants object whose target is already an object in $\mathcal{COV}$, i.e.\@ the smallest class satisfying (C1--4). 
Theorem~\ref{MAINTHEOREMEXACT}, together with Corollary~\ref{COREXACT}, shows that $\mathcal{W}_{loc} = \mathcal{W}_{\mathcal{COV}}$ in this case.
However, the characterization is self-referential and thus gives no concrete description of $\mathcal{W}_{loc}$.
\end{PAR}

\begin{PAR}
Consider the case of the left Bousfield localization at all hypercovers (cf.\@ \ref{EX2}). In this case, we may define the smallest class $\mathcal{COV}$ that contains
all representable constant presheaves and inductively all morphisms in $\Fib \cap \mathcal{W}_{loc}$ between cofibrants object whose target is already an object in $\mathcal{COV}$, i.e.\@ the smallest class satisfying (C1--4). 
It other words, $\mathcal{COV}$ contains the objects in $\mathcal{S}$, their hypercovers (which are also global fibrations) and refinements between those. 
Theorem~\ref{MAINTHEOREMEXACT}, together with Corollary~\ref{COREXACT}, shows that $\mathcal{W}_{loc} = \mathcal{W}_{\mathcal{COV}}$ in this case.
Alternatively one can enlarge $\mathcal{COV}$ to contain {\em all} hypercovers because this class still satisfies (C3) and (C4). 

However, $\mathcal{W}_{loc}$ may be also be described simply as $\mathcal{W}_{\mathcal{C}}$ (still by Definition~\ref{DEFWCOV}) where $\mathcal{C}$ is the subcategory with morphisms $\coprod h_{V_i} \rightarrow \coprod h_{U_j}$ induced by refinements of coverings  $\{U_i \rightarrow X\}$ and  $\{V_j \rightarrow X\}$  of $X \in \mathcal{S}$  (of course this class does not satisfy (C3--4)). This is well-known (cf.\@ \cite{Jar87}) and will not be reproven here. 
\end{PAR}

\appendix

\section{Local homotopy lifting}

\begin{PAR}
Let $\mathcal{S}$ be a small category and fix the projective model structure on $\mathcal{SET}^{\mathcal{S}^{\op} \times \Delta^{\op}}$ (simplicial presheaves). 
\end{PAR}

\begin{PAR}
Fix a (non-full) subcategory $\mathcal{COV}$ in $\mathcal{SET}^{\mathcal{S}^{\op} \times \Delta^{\op}}$ such that the objects 
are cofibrant in the projective model structure (which are in particular degree-wise coproducts of retracts of representables). Furthermore assume that $\mathcal{COV}$ contains all constant representable presheaves, i.e.\@ assume axioms (C1) and (C2) of \ref{PARAXIOMS}. Axioms (C3) and (C4) will not play any role in this appendix. 
\end{PAR}

\begin{BEISPIEL}
Examples (for $\mathcal{S}$ being equipped with a Grothendieck pre-topology): 
\begin{enumerate}
\item Hypercovers/bounded hypercovers of varying $X \in \mathcal{S}$ and their refinements;
\item Hypercovers of varying $X \in \mathcal{S}$ which are also \v{C}ech weak equivalences and their refinements;
\item Morphisms of the form $\coprod h_{V_i} \rightarrow \coprod h_{U_j}$ for usual refinements of coverings $\{V_i\} \rightarrow\{U_j\} \rightarrow X$.
\end{enumerate}
\end{BEISPIEL}

\comment{One could also discuss the stability of cofibrations between fibrant objects with the local homotopy lifting property. It could be that those are closed under transfinite compositions if we regard only normal coverings. But still not clear whether this is closed under all operations and hence all trivial cofibrations have that property. How can this be shown otherwise? In DHI the argument uses that $\mathcal{W}_{\mathcal{COV}'}$ formed w.r.t.\@ normal covers defines the structure in a model category in particular $\Cof \cap \mathcal{W}_{\mathcal{COF}'}$ has the usual closure properties. From this, since hypercovers are in $\mathcal{W}_{\mathcal{COV}'}$ they satisfy the usual closure properties. }

\comment{Maybe could use a Lemma that local lifting w.r.t.\@ usual covers and hypercovers is equivalent. Cech covers are somehow different... because the induction does not
go through. Lifting w.r.t.\@ bounded hypercovers seems weaker though. }

\begin{DEF}\label{DEFLOCFIB}
We say that a morphism $f$ is a local fibration if in every square
\[ \xymatrix{
\Lambda_{n,k} \otimes X  \ar[r]^{}  \ar[d] & A \ar[d]^f \\
\Delta_n \otimes X    \ar[r]^{} & B 
}\]
in which $X$ is an object in $\mathcal{COV}$
there is a morphism $X' \rightarrow X$ in $\mathcal{COV}$ and a morphism $h$ in the diagram
\[ \xymatrix{
 \Lambda_{n,k} \otimes  X'  \ar[r]^{} \ar[d] &  \Lambda_{n,k} \otimes X   \ar[r]^{}  \ar[d] & A \ar[d]^f \\
 \Delta_n \otimes X'  \ar[r]^{} \ar@{-->}[rru]^h &  \Delta_n \otimes X   \ar[r]^{} & B 
}\]
making the upper and lower triangle commute. 
\end{DEF}

Note that a global fibration is in particular a local fibration. 

\begin{DEF}\label{DEFHLP}
We say that a morphism $f$ has the local (homotopy) lifting property if in every square
\[ \xymatrix{
\partial \Delta_n \otimes X   \ar[r]^{}  \ar[d] & A \ar[d]^f \\
\Delta_n \otimes X  \ar[r]^{} & B 
}\]
where $X$ is an object in $\mathcal{COV}$
there is a morphism $X' \rightarrow X$ in $\mathcal{COV}$ and a morphism $h$ in the diagram
\[ \xymatrix{
  \partial \Delta_n \otimes X' \ar[r]^{} \ar[d] &   \partial \Delta_n \otimes X  \ar[r]^{}  \ar[d] & A \ar[d]^f \\
  \Delta_n \otimes X' \ar[r]^{} \ar@{-->}[rru]^h &  \Delta_n \otimes X   \ar[r]^{} & B 
}\]
making the upper triangle commute and making the lower triangle commute (resp.\@ commute up to left homotopy). 
\end{DEF}

\comment{
\begin{LEMMA}
If trivial fibrations are in $\mathcal{COV}$ then 
in the definition of local (homotopy) lifting one can also assume $X$ and $X'$ cofibrant. 
\end{LEMMA}
\begin{proof}
Consider a square
\[ \xymatrix{
Y \otimes \partial \Delta_n \ar[r]^{}  \ar[d] & A \ar[d]^f \\
Y \otimes \Delta_n  \ar[r]^{} & B 
}\]
and a cofibrant replacement $Y \rightarrow X$ (which is a trivial fibration). Then the  outer square in 
\[ \xymatrix{
Y \otimes \partial \Delta_n \ar[r]^{} \ar[d] &  X \otimes \partial \Delta_n \ar[r]^{}  \ar[d] & A \ar[d]^f \\
Y \otimes \Delta_n  \ar[r]^{} & X \otimes \Delta_n  \ar[r]^{} & B 
}\]
satisfies the assumption that $Y$ is cofibrant and thus a local lift exists. 
If trivial fibrations qualify as covers, we may assume that the cover is cofibrant as well. 
\end{proof}

From now on we assume that $\mathcal{COV}$ consists of morphisms between cofibrant objects. If this is not the case but $\mathcal{COV}$ contains the trivial fibrations
we may w.l.o.g.\@ downsize $\mathcal{COV}$ such that it is true. 

}

\begin{LEMMA}\label{LEMMAFINITESSET1}
 If $f$ has the (homotopy) local lifting property and $K \hookrightarrow L$ is an inclusion of finite simplicial sets then also each square
\[ \xymatrix{
K \otimes Y \ar[r]^{}  \ar[d] & A \ar[d]^f \\
L \otimes Y  \ar[r]^{} & B 
}\]
with $Y \in \mathcal{COV}$ has a local (homotopy) lifting in the obvious sense. For the homotopy case assume that $A$ and $B$ are locally fibrant. 

If $f$ has the local lifting property then also
\[ \boxdot \Hom(K \hookrightarrow L, f) \]
has it.  
\end{LEMMA}
\begin{proof}
Since $K \hookrightarrow L$ is a finite composition of push-outs of the form 
$\partial \Delta_n \rightarrow \Delta_n$, the first assertion follows by induction. 

For the second assertion note that the local lifting property for $\boxdot \Hom(K \hookrightarrow L, f)$ is equivalent to the existence of a local lifting in the diagram
\[ \xymatrix{
 (L \times \partial \Delta_n \cup K \times \Delta_n ) \otimes  Y \ar[r]^{}  \ar[d] & A \ar[d]^f \\
 (L \times \Delta_n) \otimes Y   \ar[r]^{} & B 
}\]
If $K \hookrightarrow L$ is an inclusion of finite simplicial sets then also $(K \hookrightarrow L) \boxplus  (\partial \Delta_n \rightarrow \Delta_n)$ is. 
\end{proof}

\begin{LEMMA}\label{LEMMAFINITESSET2}
If $f$ is a local fibration and $K \hookrightarrow L$ is a strong anodyne extension of finite simplicial sets then each square
\[ \xymatrix{
K \otimes Y \ar[r]^{}  \ar[d] & A \ar[d]^f \\
L \otimes Y  \ar[r]^{} & B 
}\]
with $Y \in \mathcal{COV}$  has a local lifting. 

If $f$ is a local fibration and $K \hookrightarrow L$ is a strong anodyne extension of finite simplicial sets
\[ \boxdot \Hom(K \hookrightarrow L, f) \]
is a local fibration. 
\end{LEMMA}
\begin{proof}
The first assertion follows by the same proof as \cite[1.4]{Jar87}.
Furthermore, if $K \hookrightarrow L$ is a a strong anodyne extension of finite simplicial sets extension then also $(K \hookrightarrow L) \boxplus  (\Lambda_{n,k} \rightarrow \Delta_n)$ is by \cite[1.3]{Jar87}. 
\end{proof}

\comment{
\begin{LEMMA}\label{LEMMAW}
A local fibration which is a weak equivalence has the local lifting property.
\end{LEMMA}
\begin{proof}

First: A trivial cofibration $f: X \rightarrow Y$ between locally fibrant objects has the homotopy local lifting property: 
Consider
\[ \boxdot \Hom(\partial \Delta_n \rightarrow \Delta_n , X \rightarrow Y) \]

\[ \xymatrix{
 \Hom(\Delta_n,X) \ar[d] \\
 \Box \ar[r]^{}  \ar[d] & \Hom(\partial \Delta_n, X) \ar[d] \\
 \Hom( \Delta_n, Y)  \ar[r]^{} &  \Hom(\partial \Delta_n, Y)
}\]

Hence it is a weak equivalence if those are stable under pull-back along local fibrations.

NEW: 
Between fibrant objects a weak equivalence has the homotopy lifting property (w/o) local.
$X \rightarrow Y \rightarrow Y'$ where $Y'$ is fibrant and  

\end{proof}
}

\begin{LEMMA}\label{LEMMALOCFIBHOMOTOPY}
For a local fibration local homotopy lifting and local lifting are equivalent
\end{LEMMA}
\begin{proof}
Consider a homotopy local lifting such that the lower triangle commutes via the homotopy $\mu: fh \Rightarrow a\iota_1$:
\[ \xymatrix{
\partial \Delta_n \otimes W'   \ar[r]^{}  \ar[d] &\partial \Delta_n \otimes W   \ar[r]^{}  \ar[d] & X \ar[d]^f \\
 \Delta_n \otimes W'  \ar@{-->}[rru]^h \ar[r]_-{\iota_1} & \Delta_n \otimes W  \ar[r]_-{a} & Y
}\]
Then consider

\[ \xymatrix{
  (\Delta_n \times \{0\} \cup \partial \Delta_n \times \Delta_1) \otimes W' \ar[rr]^-{h, c_{a}\iota_1} \ar[d]  & & X \ar[d]^f  \\
 (\Delta_n \times \Delta_1) \otimes W'  \ar[rr]^-{\mu} & & Y
 }\]
Since $f$ is a local fibration, by Lemma~\ref{LEMMAFINITESSET2}, there is $\iota_2: W'' \rightarrow W'$ in $\mathcal{COV}$ and a lift  
\[ \xymatrix{
&  (\Delta_n \times \{0\} \cup \partial \Delta_n \times \Delta_1) \otimes W'' \ar[rr]^-{h, c_{a}\iota_1\iota_2} \ar[d]  & & X \ar[d]^f  \\
(\Delta_n \times \{1\}) \otimes W'' \ar[r]^{e_0} & (\Delta_n \times \Delta_1) \otimes W''  \ar@{-->}[rru] \ar[rr]^-{\mu \iota_2} & & Y
 }\]
The composition with the left horizontal morphism $e_0$ is then the lift which makes
 everything commute on the nose. 
\end{proof}

\begin{LEMMA}\label{LEMMAPULLBACK}
Consider a pull-back square
\[ \xymatrix{ X' \ar[r]  \ar[d]_{f'} & X  \ar[d]^f \\
Y' \ar[r] & Y} \]
If $f$ has the local lifting property (respectively is a local fibration) then also $f'$ has the local lifting property (respectively is a local fibration). 
\end{LEMMA}
\begin{proof}Obvious.
\end{proof}

\begin{LEMMA}\label{LEMMA2OUTOF3}
\begin{enumerate}
\item If $f$ and $g$ have the (homotopy) lifting property then also $gf$ has it.
\item if $gf$ and $f$ have the (homotopy) lifting property then also $g$ has it.
\item if $gf$ and $g$ have the (homotopy) lifting property then also $f$ has it.
\end{enumerate}

For the statements involving ``homotopy'' assume that $X, Y$ and $Z$ are locally fibrant. For assertion 3.\@ without ``homotopy'' assume that $f$ is a local fibration.
\end{LEMMA}
\begin{proof}
1.
Consider a diagram 
\[ \xymatrix{ 
 & X \ar[d]^f \\
\partial \Delta_n \otimes W  \ar[ru]  \ar[d] & Y \ar[d]^g \\
\Delta_n \otimes W \ar[r]^a & Z  \\
} \]
Applying the assumption, we get a local lift
\[ \xymatrix{ 
 & & X \ar[d]^f \\
\partial \Delta_n \otimes W'  \ar[r] \ar[d] & \partial \Delta_n \otimes W  \ar[ru] \ar[r] \ar[d] & Y \ar[d]^g \\
\Delta_n \otimes W' \ar[r]^{\iota_1}  \ar@{-->}[rru]^h & \Delta_n \otimes W \ar[r]^a & Z  \\
} \]
and a local lift
\[ \xymatrix{ 
\partial \Delta_n \otimes W'' \ar[r]\ar[d]   & \partial \Delta_n \otimes W' \ar[r]   \ar[d] & X \ar[d]^f \\
\partial \Delta_n \otimes W'' \ar[r]^{\iota_2} \ar@{-->}[rru]^{h'}   & \Delta_n \otimes W' \ar[r]^h & Y  \\
} \]
with homotopies $ gfh' \Rightarrow g h \iota_2 \Rightarrow  a \iota_1 \iota_2$. The homotopies may be composed (refining the cover if necessary) because $Z$ is locally fibrant. 
If the homotopies are equalities then we do not have to assume anything.

2.
Consider a diagram 
\[ \xymatrix{ 
 & X \ar[d]^f \\
\partial \Delta_n \otimes W  \ar[r]  \ar[d] & Y \ar[d]^g \\
\Delta_n \otimes W \ar[r]_-a & Z  \\
} \]
By Lemma~\ref{LEMMAFINITESSET1} applied to the morphism $\emptyset \rightarrow \partial \Delta_n$, we get a lifting
\[ \xymatrix{ 
 & & X \ar[d]^f \\
\partial \Delta_n \otimes W'  \ar[d] \ar[r] \ar@{-->}[rru]^h & \partial \Delta_n \otimes W  \ar[r]  \ar[d] & Y \ar[d]^g \\
\Delta_n \otimes W' \ar[r]^{\iota_1} &  \Delta_n \otimes W \ar[r]^a & Z  \\
} \]
and a homotopy $\mu: gfh \Rightarrow a \iota_1$ defined on $\partial \Delta_n \otimes W'$. 

If this homotopy is not trivial, and $Z$ is locally fibrant, consider:

\[ \xymatrix{
  (\Delta_n \times \{1\} \cup \partial \Delta_n \times \Delta_1) \otimes W' \ar[rr]^-{a \iota_1, \mu} \ar[d]  & & Z  \\
 (\Delta_n \times \Delta_1) \otimes W'  & &
 }\]
By Lemma~\ref{LEMMAFINITESSET2}, there is $\iota_2: W'' \rightarrow W'$ in $\mathcal{COV}$ and a lift  
\[ \xymatrix{
&  (\Delta_n \times \{1\} \cup \partial \Delta_n \times \Delta_1) \otimes W'' \ar[rr]^-{a \iota_1 \iota_2, \mu \iota_2} \ar[d]  & & Z   \\
(\Delta_n \times \{0\}) \otimes W'' \ar[r]^{e_0} & (\Delta_n \times \Delta_1) \otimes W''  \ar@{-->}[rru]^{h'}  & &
 }\]

Define $a' := h' e_0$. We have a commutative diagram and a local homotopy lift 
\[ \xymatrix{ 
\partial \Delta_n \otimes W''' \ar[r] \ar[d]   & \partial \Delta_n \otimes W'' \ar[r]^-h  \ar[d] & X \ar[d]^{gf} \\
\partial \Delta_n \otimes W''' \ar[r]^{\iota_3} \ar@{-->}[rru]^{h'}   & \Delta_n \otimes W'' \ar[r]^-{a'} & Z  \\
} \]
In total we have homotopies: 
\[ g f h' \Rightarrow  a'  \iota_3 \Rightarrow    a  \iota_1 \iota_2 \iota_3. \]
Those may be composed refining the cover because $Z$ is locally fibrant. In case that they are identities we do not have to assume anything.

3.\@ 
Consider a diagram 
\[ \xymatrix{ 
 \partial \Delta_n \otimes W  \ar[r]  \ar[d] & X \ar[d]^f \\
\Delta_n \otimes W  \ar[r]_-a & Y \ar[d]^g \\
 & Z  \\
} \]
We get a lifting
\[ \xymatrix{ 
 \partial \Delta_n \otimes W'  \ar[r]  \ar[d] &\partial \Delta_n \otimes W  \ar[r]  \ar[d] & X \ar[d]^f \\
\Delta_n \otimes W'  \ar[r]^{\iota_1} \ar[rrd] \ar@{-->}[rru]^h & \Delta_n \otimes W  \ar[r]^-{a} & Y \ar[d]^g \\
 & & Z  \\
} \]
and a homotopy $\mu: gfh \Rightarrow ga\iota_1$. Then from the diagram

\comment{ 

If this is not trivial we assumed that $Z$ is locally fibrant and consider:

\[ \xymatrix{
  (\Delta_n \times \{1\} \cup \partial \Delta_n \times \Delta_1) \otimes W' \ar[rr]^-{g a \iota_1, \mu} \ar[d]  & & Z  \\
 (\Delta_n \times \Delta_1) \otimes W'  & &
 }\]
Using Lemma~\ref{LEMMAFINITESSET2}, there is $\iota_2: W'' \rightarrow W'$ in $\mathcal{COV}$ and a lift  
\[ \xymatrix{
&  (\Delta_n \times \{1\} \cup \partial \Delta_n \times \Delta_1) \otimes W'' \ar[rr]^-{g a \iota_1 \iota_2, \mu \iota_2} \ar[d]  & & Z   \\
(\Delta_n \times \{0\}) \otimes W'' \ar[r]^{e_0} & (\Delta_n \times \Delta_1) \otimes W''  \ar@{-->}[rru]^{h'}  & &
 }\]

and define $a' := h' e_0$. Then consider:

\[ \xymatrix{ 
(\partial \Delta_1 \times \Delta_n) \otimes W''' \ar[r] \ar[d]   & (\partial \Delta_1 \times \Delta_n) \otimes W'' \ar[rr]^-{ f h \iota_1 \iota_2, a \iota_1 \iota_2 }  \ar[d] &&  Y  \ar[d]^{g} \\
( \Delta_1 \times \Delta_n ) \otimes W''' \ar[r]^{\iota_3} \ar@{-->}[rrru]^{h''}   & ( \Delta_1 \times \Delta_n ) \otimes W'' \ar[rr]^-{h'}  & & Z  \\
} \]

we get a homotopy on $W'''$:
\[   f h \iota_1 \iota_2 \iota_3  \Rightarrow a \iota_1 \iota_2 \iota_3   \]

If $f$ is a local fibration, then this homotopy may be used to change $h$ such that the thing commutes on the nose (up to refining the cover). 

NEXT TRY: }
\[ \xymatrix{ 
(\partial \Delta_1 \times \Delta_n) \otimes W'' \ar[r] \ar[d]   & (\partial \Delta_1 \times \Delta_n) \otimes W' \ar[rr]^-{ f h  , a \iota_1 }  \ar[d] &&  Y  \ar[d]^{g} \\
( \Delta_1 \times \Delta_n ) \otimes W'' \ar[r]_{\iota_2} \ar@{-->}[rrru]^(.4){h'}   & ( \Delta_1 \times \Delta_n ) \otimes W' \ar[rr]_-{\mu}  & & Z  \\
} \]
we get a homotopy on $\Delta_n \otimes W''$:
\[   f h \iota_1 \iota_2  \Rightarrow a \iota_1 \iota_2.   \]
If $f$ is a local fibration, then as in Lemma~\ref{LEMMALOCFIBHOMOTOPY} --- up to refining the cover --- this homotopy may be used to change $h$ such that the diagram commutes on the nose. 
\end{proof}

\begin{LEMMA}\label{LEMMATFIBLIFTING}
Let $f, f'$ be trivial cofibrations and  $g, g'$ be (global) fibrations satisfying $gf = g'f'$. Then $g$ has the local lifting property if and only if $g'$ has the local lifting property.
\end{LEMMA}
\begin{proof}
Let $g$ have the local lifting property and
form the pull-back
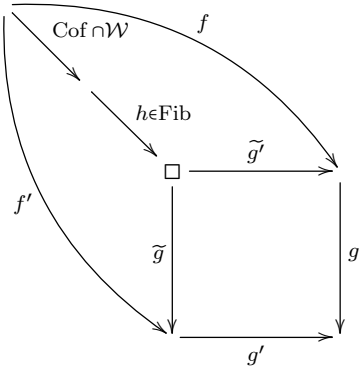
\[ \xymatrix{ 
 \ar[rd]^{\Cof \cap \mathcal{W}} \ar@/^20pt/[rrrrdd]^{f} \ar@/_20pt/[rrdddd]_{f'}  \\
&  \ar[rd]^{h \in\Fib}  \\
& & \Box \ar[rr]^{\widetilde{g'}} \ar[dd]_{\widetilde{g}} & &   \ar[dd]^g \\
\\
& &  \ar[rr]_{g'} & & 
} \]
Since the local lifting property is preserved under pull-back (Lemma~\ref{LEMMAPULLBACK}),  $\widetilde{g}$ has the local lifting property. 
Now $\widetilde{g} h$ and $\widetilde{g}' h$ are trivial fibrations and thus have the lifting property.
Therefore by 2-out-of-3 also
$h$ has the local lifting property. Therefore also $\widetilde{g}'$ has the lifting property.
Again by 2-out-of-3 $g'$ has it as well because all 3 other fibrations in the square have. 
\end{proof}

\newpage
\bibliographystyle{abbrvnat}
\bibliography{6fu}

\end{document}